\documentclass{amsart}
\usepackage{amssymb}

\newtheorem{thm}{Theorem}[section]
\newtheorem{cor}[thm]{Corollary}
\newtheorem{lem}[thm]{Lemma}
\newtheorem{prop}[thm]{Proposition}
\theoremstyle{definition}

\newtheorem{rem}[thm]{Remark}
\numberwithin{equation}{section}

\begin{document}
\title[clique number]{On the clique number of  non-commuting graphs of certain groups}%
\author{A. Abdollahi, A. Azad, A. Mohammadi Hassanabadi and M. Zarrin }%
\address{$^\mathbf{1}$Department of Mathematics, University of Isfahan, Isfahan 81746-73441, Iran}%
\address{$^\mathbf{2}$School of Mathematics, Institute for Research in Fundamental Sciences (IPM), P.O.Box: 19395-5746, Tehran, Iran.}%
\address{$^\mathbf{3}$ Shaikhbahaee University, Isfahan 81797-35296, Iran.}
\email{($^\mathbf{1,2}$A. Abdollahi) \;\; a.abdollahi@math.ui.ac.ir \;\ abdollahi@member.ams.org}%
\email{($^\mathbf{1}$A. Azad) \;\; a-azad@sci.ui.ac.ir}%
\email{($^\mathbf{1,3}$A. Mohammadi Hassanabadi) \;\; aamohaha@yahoo.com}%
\email{($^\mathbf{1}$M. Zarrin) \;\; m.zarrin@math.ui.ac.ir}%
\begin{abstract} Let $G$ be a non-abelian group. The non-commuting graph $\mathcal{A}_G$ of $G$
is defined as the graph whose vertex set is the non-central
elements of $G$ and two vertices are joint if and only if they do
not commute. In a finite simple graph $\Gamma$ the maximum size
of a complete subgraph of $\Gamma$ is called the clique number of
$\Gamma$ and it is denoted by $\omega(\Gamma)$. In this paper we
characterize all non-solvable groups $G$
 with $\omega(\mathcal{A}_G)\leq 57$, where the number $57$ is the clique number of the non-commuting graph of the projective special linear group $\mathrm{PSL}(2,7)$.
 We also complete the determination of  $\omega(\mathcal{A}_G)$ for all finite minimal simple
 groups.
\end{abstract}
\maketitle
\section{Introduction and results}
Let $G$ be a non-abelian  group and $Z(G)$ be its center.
Following \cite{AAM} and \cite{Mogh}, the non-commuting graph
$\mathcal{A}_G$ of $G$ is defined as the graph whose  vertex set
is $G\backslash Z(G)$ and two distinct vertices $a$ and $b$ are
joint whenever $ab\not=ba$. Let $\Gamma$ be a simple graph. The
set of vertices of every complete subgraph of $\Gamma$ is called a
clique of $\Gamma$. The maximum size (if it exists) of a complete
subgraph of $\Gamma$ is called the clique number of $\Gamma$ and
it is denoted by $\omega(\Gamma)$. Thus a clique  of $\mathcal{A}_G$
is no more than a set of pairwise non-commuting elements of $G$.
However, as the following results show, the clique number of the
non-commuting graph of a group not only has some influence on the
structure of a group but also finding it, is important in some
areas  such as cohomology ring of a group.
 By a famous
result of Neumann \cite{N} answering a question of P. Erd\"os, we
know that  the finiteness of all cliques in $\mathcal{A}_G$
implies the finiteness of the factor group $G/Z(G)$ (and so
$\omega(\mathcal{A}_G)$ is finite). In \cite[Theorem 1.4]{AM},
non-solvable groups $G$ satisfying the condition
$\omega(\mathcal{A}_G)\leq 21$ are characterized. Specifically, such a
group $G$ is isomorphic to $Z(G) \times A_5$, where $A_5$ is the
alternating group of degree $5$. Also according to \cite[Theorem
1.5]{AM}, the derived length  of a non-abelian solvable group $G$
is at most $2\omega(\mathcal{A}_G)-3$.

For a prime number $p$,  a finite $p$-group $G$ is called
extra-special if the center, the Frattini subgroup and the
derived subgroup of $G$ all coincide and are cyclic of order $p$.
 The clique number of extra-special $p$-groups is important as it provides
  combinatorial information which can be used to calculate their cohomology lengths (The cohomology length of a non-elementary abelian
  $p$-group  is a cohomology invariant   defined as a result
  of a Serre's theorem \cite{S}). Chin \cite{Chin} has obtained upper and lower bounds for
clique numbers of  non-commuting graphs of extra-special
$p$-groups, for odd prime numbers $p$. Specifically, it is proved
in \cite[Theorem 2.2]{Chin} that if $G_n$ is an extra-special
group of order $p^{2n+1}$, then $\omega(\mathcal{A}_{G_1})=p+1$ and
$$np+1\leq \omega(\mathcal{A}_{G_{n+1}})\leq \frac{p(p-1)^n-2}{p-2}.$$
For $p=2$, it has been shown by Isaacs (see \cite[p. 40]{B}) that
$\omega(\mathcal{A}_G)=2m+1$ for any
extra-special group $G$ of order $2^{2m+1}$.\\

Finding the clique number of the non-commuting graph of a group
itself is   of independent interest as a pure combinatorial
problem. Brown in \cite{Br1} and \cite{Br2} has studied the
clique and chromatic numbers of $\mathcal{A}_{S_n}$, where $S_n$
is the symmetric group  of degree $n$. It is proved in \cite{Br2}
that $\omega(\mathcal{A}_{S_n})\not=\chi(\mathcal{A}_{S_n})$ for all
$n\geq 15$, where $\chi(\Gamma)$ is the chromatic number of the
graph $\Gamma$. It is easy to see that the  chromatic number of
$\mathcal{A}_G$ is equal to the minimum number (if it exists) of
abelian subgroups of $G$ whose set-theoretic union is the whole
group $G$.\\
In \cite{AM} the authors have determined non-solvable groups $G$ with $\omega(\mathcal{A}_G)\leq 21$, where the number $21$ is the clique number of the non-commuting graph of the least (with respect to the order) non-abelian simple group $A_5$. The clique number of the non-commuting graph of $\mathrm{PSL}(2,7)$ (which is  the second-least order non-abelian simple group)  is $57$. Here we give a characterization of non-solvable groups $G$ with $\omega(\mathcal{A}_G)\leq 57$.
\begin{thm}\label{1.3}
Let $G$ be a finite non-solvable group such that $\omega(\mathcal{A}_G)\leq
57$. Then $G$ has  one of the following structures
\begin{enumerate}
\item  $G\cong Z(G)\times \mathrm{PSL}(2,p)$, where $p\in\{5,7\}$;
 \item  $G=Z(G)K$, where $K$ is a subgroup of $G$ isomorphic to $\mathrm{SL}(2,p)$ and $p\in\{5,7\}$;
 \item $G=G''\langle a\rangle S$, where $a^2\in Z(G)$ and $G''\cong A_5$ or
$\mathrm{SL}(2,5)$, and $S$ is the solvable radical of $G$;
\item $G=G''\langle a\rangle Z(G)$, where $a^2\in Z(G)$ and $G''\cong \mathrm{PSL}(2,7)$ or
$\mathrm{SL}(2,7)$.
\end{enumerate}
\end{thm}
In \cite[Lemma 4.4]{AAM},  the clique numbers of $\mathcal{A}_G$ of
all projective special linear groups $G=\mathrm{PSL}(2,q)$ have
been obtained. A family of the  minimal simple groups (i.e.
finite non-abelian simple groups all of whose  proper subgroups are
solvable) are projective special linear groups of degree 2 over a
finite field. All minimal simple groups  were completely
classified by a well-known result of Thompson \cite{T}. In
Section 3, we shall find the clique number of $\mathcal{A}_G$ for
the remaining finite minimal simple groups $G$ i.e., the Suzuki
groups $\mathrm{Sz}(2^{2m+1})$ and the projective special linear
group $\mathrm{PSL}(3,3)$ over the field with 3 elements. As
Thompson's classification of the minimal simple groups is a very
useful tool to obtain solvability criteria  in the class of
finite groups (see \cite{W} for a recent and interesting
application of Thompson's theorem), we hope that these new
information  might be useful to obtain new solvability
criterion.\\
A clique of a graph $\Gamma$ is called a maximum clique if its
size is $\omega(\Gamma)$. We say that  a clique $X$ of a graph
can be extended to a maximum clique if there exists a maximum
clique containing $X$. We will prove that every clique of
$\mathcal{A}_G$ for every minimal simple group $G$ except
$\mathrm{PSL}(3,3)$ can be extended to a maximum clique of
$\mathcal{A}_G$ (see Proposition \ref{50} and Theorem \ref{1.4}).
\begin{thm}\label{1.4} Let $G=\mathrm{Sz}(q)$ ($q=2^{2m+1}$ and $m>0$) be the Suzuki
group over the field with $q$ elements (see \cite[p. 182]{HB}).
Then
\begin{enumerate}
\item $\omega(Sz(q))=(q^2+1)(q-1)+\frac{q^2(q^2+1)}{2}+\frac{q^2(q^2+1)(q-1)}{4(q+2r+1)}
+\frac{q^2(q^2+1)(q-1)}{4(q-2r+1)}$, where $r=2^m$.
\item Every clique of  $\mathcal{A}_G$ can be extended to a
maximum clique  of $\mathcal{A}_G$.
\end{enumerate}
\end{thm}
\begin{thm}\label{26}
 $\omega(\mathcal{A}_{\mathrm{PSL}(3,3)})=1067$.
\end{thm}
We use the usual notation: for example $C_G(a)$ is the centralizer
of an element $a$ in a group $G$, $N_G(H)$ is the normalizer of a
subgroup $H$ in $G$, $\mathrm{GL}(n,q)$, $\mathrm{SL}(n,q)$,
$\mathrm{PGL}(n,q)$ and $\mathrm{PSL}(n,q)$ denote respectively,
the general linear group, the special linear group, the projective
general linear group, and the projective special linear group of
degree $n$ over the finite field of order $q$, and $D_{2n}$ is the
dihedral group of order $2n$. A family $\{G_1,\dots,G_k\}$ of
proper subgroups  of a group $G$ is called a partition of $G$ if
every non-identity element of $G$ belongs to exactly one of
the $G_i$'s.
\section{Proofs}
To prove Theorem \ref{1.3} we need the following lemmas.
\begin{lem}\label{1} Let $G$ be a finite non-abelian group.
\begin{enumerate}
\item[(i)] For any non-abelian subgroup $H$ of $G$, $\omega(\mathcal{A}_H)\leq \omega(\mathcal{A}_G)$.
\item[(ii)] For any non-abelian  factor group $G/N$ of $G$,
$\omega(\mathcal{A}_{\frac{G}{N}})\leq \omega(\mathcal{A}_G)$.
\end{enumerate}
\end{lem}
\begin{proof}
It is straightforward.
\end{proof}
\begin{lem}
$\omega(\mathcal{A}_{PGL(2,q)})=\begin{cases}
  4 & \hbox{if   $q=2$} \\
  10 & \hbox{if   $q=3$} \\
   q^2+q+1 & \hbox{if   $q>3$}
\end{cases}$
\end{lem}
\begin{proof}By Lemma \ref{1}, we have $$\omega(\mathcal{A}_{\mathrm{PSL}(2,q)})\leq
\omega(\mathcal{A}_{\mathrm{PGL}(2,q)})\leq
\omega(\mathcal{A}_{\mathrm{GL}(2,q)}).$$ Now, if $q>5$ or $q=4$, then
by \cite[Lemma 4.4]{AAM}
 $$q^2+q+1 = \omega(\mathcal{A}_{\mathrm{PSL}(2,q)}) \leq
\omega(\mathcal{A}_{\mathrm{PGL}(2,q)}) \leq
\omega(\mathcal{A}_{\mathrm{GL}(2,q)})= q^2+q+1.$$ Thus
$\omega(\mathcal{A}_{\mathrm{PGL}(2,q)})=q^2+q+1$, for $q=4$ and
$q>5$. Also since $\mathrm{PGL}(2,5)\cong S_5$ and
$\mathrm{PGL}(2,3)\cong S_4$  it follows from  \cite[p. 2]{Br1}
that $\omega(\mathcal{A}_{\mathrm{PGL}(2,5)})=31$,
 $\omega(\mathcal{A}_{\mathrm{PGL}(2,3)})=10$ and as $\mathrm{PGL}(2,2)\cong\mathrm{PSL}(2,2)$, by \cite[Lemma 4.4]{AAM} we have
 that $\omega(\mathcal{A}_{\mathrm{PGL}(2,2)})=4$. This completes the
 proof.
\end{proof}
\begin{thm}\label{2} Let $G$ be a  non-abelian simple group such that  $\omega(\mathcal{A}_G) \leq
57$. Then $G \cong A_5$ or $G \cong \mathrm{PSL}(2,7)$.
\end{thm}
\begin{proof} By Neumann's result \cite{N}, $G/Z(G)$ is finite and since $G$ is a non-abelian
simple group, we have that $Z(G)=1$. Thus $G$ is finite.
 Suppose that the result is false, and let $M$ be a minimal counter example.
Thus every proper non-abelian simple section of $M$ is isomorphic
to $A_5$ or $\mathrm{PSL}(2,7)$. By \cite[Proposition 4]{BR} $M$
is isomorphic to one of the following:\\
 $\mathrm{PSL}(2,2^m)$, $m=4$ or a prime;\\
 $\mathrm{PSL}(2,3^p)$, $\mathrm{PSL}(2,5^p)$, $\mathrm{PSL}(2,7^p)$, $p$ a prime;\\
 $\mathrm{PSL}(2,p)$, $p>7$;\\
 $\mathrm{PSL}(3,3)$, $\mathrm{PSL}(3,5)$, $\mathrm{PSL}(3,7)$;\\
  $\mathrm{PSU}(3,3)$, $\mathrm{PSU}(3,4)$, $\mathrm{PSU}(3,7)$ (the projective special unitary group of degree $3$
  over the finite field of order $3$,$4$ and $7$ respectively) or\\
  $\mathrm{Sz}(2^p)$, $p$ an odd prime.\\
  Now, for every prime number $p$ and every integer $n \geq 0$, by
  \cite[Lemma 4.4]{AAM}, $\omega(\mathcal{A}_{\mathrm{PSL}(2,p^n)})=p^{2n}+p^n+1$.
Thus since $\mathrm{PSL}(2,2^2)\cong A_5$, among the projective
special linear groups, we only need to investigate
 $\mathrm{PSL}(3,3)$, $\mathrm{PSL}(3,5)$ and $\mathrm{PSL}(3,7)$.
For each prime divisor $p$ of $|G|$, let $\nu_p(G)$ be the number
of  Sylow $p$-subgroups of $G$. If $p$ is a prime number dividing
$|G|$ such that the intersection of any two distinct Sylow
$p$-subgroups is trivial, then by \cite[Lemma 3]{E}, we must have  $\nu_p(G) \leq 57$ (*).\\
Now $\mathrm{PSL}(3,3)$ has order $2^4 \times 3^3 \times 13$, so
$\nu_{13}(\mathrm{PSL}(3,3))>57$.\\
$\mathrm{PSL}(3,5)$ has order $2^5 \times 3 \times 5^3 \times 31$,
so $\nu_{31}(\mathrm{PSL}(3,5))>57$.\\
$\mathrm{PSL}(3,7)$ has order $2^5 \times 3 \times 7^3 \times 19$,
so $\nu_{19}(\mathrm{PSL}(3,7))>57$.\\
$\mathrm{PSU}(3,3)$ has order $2^5 \times 3^3 \times 7$, so
$\nu_{7}(\mathrm{PSU}(3,3))>57$.\\
$\mathrm{PSU}(3,4)$ has order $2^6 \times 3 \times 5^2 \times 13$,
so $\nu_{13}(\mathrm{PSU}(3,4))>57$.\\
$\mathrm{PSU}(3,7)$ has order $2^7 \times 3 \times 7^3 \times 43$,
so $\nu_{43}(\mathrm{PSU}(3,7))=1+43k $, for some $k> 0$, and
since $44$ does not  divide  $|\mathrm{PSU}(3,7)|$, so
$\nu_{43}(\mathrm{PSU}(3,7))>57$.\\
For $p$ an odd prime, $\mathrm{Sz}(2^p)$ has order $2^{2p}\times
(2^p-1)\times(2^{2p}+1)$ and $\nu_{2}(Sz(2^p))=2^{2p}+1\geq 65$
(see \cite[chapter XI, Theorem
 $3.10$ and its proof]{H}). Now (*) completes the proof.
\end{proof}
\begin{prop}\label{11} Let $G=\mathrm{PGL}(2,q)$, where $q$ is a power of a prime number $p$ and  let
$k=\gcd(q-1,2)$. Then
\begin{enumerate}
\item a Sylow $p$-subgroup $P$ of $G$ is an elementary abelian
group of order $q$ and the number of Sylow $p$-subgroups of $G$ is
$q+1$.
\item $G$ contains a cyclic subgroup $D$ of order $q-1$ such that
the number of conjugates of $D$ is $\frac{q(q+1)}{2}$.
\item $G$ contains a cyclic subgroup $I$ of order $q+1$ such that
the number of conjugates of $I$ is $\frac{q(q-1)}{2}$.
\item The set $\{P^x,D^x,I^x \;|\; x \in G\}$ is a partition for $G$.
If $q$ is odd, then the following hold for non-trivial elements $a\in D$ and $b\in P$.
\begin{enumerate}
\item If $a$ is not of order $2$, then $C_G(a)=D$.
\item If $a$ is of order $2$, then $C_G(a)\cong D_{2(q-1)}$.
\item $C_G(b)=P$.
\end{enumerate}
\item If $q\equiv 0 $  (mod $4$), then $G=\mathrm{PGL}(2,q)\cong \mathrm{PSL}(2,q)$ and
by Proposition 3.21 of \cite{AAM}, if $a$ is a non-trivial element
of $G$, then
$$C_G(a)=\begin{cases}
 P^x & \hbox{if $a\in P^x$}\\
 D^x & \hbox{if $a\in D^x$}  \\
 I & \hbox{if $a\in I^x$}
\end{cases}$$
\end{enumerate}
\end{prop}
\begin{proof}
The proof follows from the results in Chapter II of \cite{H}
concerning projective linear groups.
\end{proof}
A group $G$ is called an AC-group if the centralizer of every
non-central element is abelian.
\begin{lem}\label{9} Let $G$ be a non-abelian AC-group  such that
$\omega(\mathcal{A}_G)$ is finite. Then every non-empty clique of
$\mathcal{A}_G$ can be extended to a maximum clique set of $G$.
\end{lem}
\begin{proof} Let $\omega(\mathcal{A}_G)=n$. Then there exist elements
$a_1,\dots, a_n$ in $G$ such that $[a_i,a_j]\neq 1$, for all
$i\neq j$. Thus $G=C_G(a_1)\cup \cdots \cup C_G(a_n)$ and also
$C_G(a_i)\cap C_G(a_j)=Z(G)$, since $G$ is an AC-group. Therefore
$\{C_G(a_i)\backslash Z(G) \;|\; i=1,\dots,n\}$ is a partition for
$G\backslash Z(G)$. Let $X$ be a clique of $\mathcal{A}_G$. Then, for each $i$,
$1\leq i\leq m$,  $|X \cap C_G(a_i)|\leq 1$, as each $C_G(a_i)$
is abelian. Let
$$I=\{i\in \{1,2,\ldots,n\} \;|\; X\cap
C_G(a_i)=\varnothing\}.$$  For each $\ell\in I$, choose an
element $b_\ell$ in $C_G(a_\ell)\backslash Z(G)$. Then $X \cup
\{b_\ell \;|\; \ell\in I\}$ is the required clique set.
\end{proof}
\begin{prop} \label{50} Let $G=\mathrm{PSL}(2,q)$ or $\mathrm{PGL}(2,q)$, where $q$ is a power of a prime $p$. Then
any singleton containing a non-central element of $G$ can be
extended to a maximum clique set of $\mathcal{A}_G$.
\end{prop}
\begin{proof}  We give only the proof for $G=\mathrm{PSL}(2,q)$, for the other group the proof is similar and
Lemma \ref{11} may be used in the proof.\\ By \cite[Proposition 3.21]{AAM}, $\mathcal{P}=\{P^x,A^x,B^x | x \in
G\}$  is a partition for $G$, where $P$ is a Sylow $p$-subgroup,
$A$ is a cyclic subgroup of order $\frac{q-1}{k}$
 and $B$ is a cyclic subgroup of order $\frac{q+1}{k}$, where $k=\gcd(q-1,2)$.
 Let $a$ be a non-trivial element of $G$. Then $a\in M$, for some
 $M\in \mathcal{P}$. Now take an arbitrary non-trivial element $b_N$ in each member $N\in\mathcal{P}$ which is different from $M$.
 Let $X$ be such a set of elements. For $q>5$,  it is not hard to see
 that $\{a\} \cup X$ is a maximum clique set for $\mathcal{A}_G$
 (see the proof of Proposition 3.21 of \cite{AAM}). For $q\leq 5$,
 see  the proof of Proposition 3.21 of \cite{AAM}.
\end{proof}
\begin{thm} \label{thm57}Let $G$ be a  semi-simple
group, such that $\omega(\mathcal{A}_G) \leq 57$. Then $G\cong
A_5$, $S_5$, $\mathrm{PSL}(2,7)$ or $\mathrm{PGL}(2,7)$.
\end{thm}
\begin{proof}
By Neumann's result \cite{N}, $G$ is finite, since in a
semi-simple group the center is trivial. Let $R$ be the
centerless CR-Radical of $G$. Then $R$ is a direct product of a
finite number of finite non-abelian simple groups, say $R\cong
S_1\times \ldots  \times S_m$. By Lemma \ref{1}, for each
$i\in\{1,\dots,m\}$, $\omega(S_i)\leq 57$. Now by Theorem
\ref{2}, for each $i\in\{1,\dots,m\}$, $S_i\cong A_5$ or $S_i\cong
\mathrm{PSL}(2,7)$. Since $\omega(\mathcal{A}_{S_i})\leq 21$,  it
follows from  \cite[Lemma 2.2]{AM} that $m=1$. Therefore $R\cong
A_5$ or $R\cong \mathrm{PSL}(2,7)$.
 We know that $C_G(R)=1$ and so $G$ is embedded into $Aut(R)$.
 If $R\cong A_5$, $Aut(R)\cong S_5$ and so $G\cong A_5$
 or $G\cong S_5$; if $R\cong \mathrm{PSL}(2,7)$, then   $Aut(R)\cong \mathrm{PGL}(2,7)$ and
 $G\cong \mathrm{PSL}(2,7)$ or $G\cong \mathrm{PGL}(2,7)$. This
 completes the proof.
\end{proof}
For a finite group $G$, $Sol(G)$  denotes the solvable radical of
$G$, i.e., the largest normal solvable subgroup of $G$.
\begin{cor}\label{7} Let $G$ be a finite group such that $\omega(\mathcal{A}_{\frac{G}{Sol(G)}})=57$. Then $\frac{G}{Sol(G)}\cong
\mathrm{PSL}(2,7)$ or $\frac{G}{Sol(G)}\cong \mathrm{PGL}(2,7)$.
\end{cor}
\begin{proof}
Since for any finite group $M$, $M/Sol(M)$ has no non-trivial and
proper normal  abelian subgroup, the proof follows from
Theorem \ref{thm57}.
\end{proof}
\begin{lem}\label{15} Let $G$ be a finite non-solvable group such that $\omega(\mathcal{A}_G) \leq 57$
 and $\frac{G}{Sol(G)}\cong A_5$. Then $G\cong Z(G)\times A_5$ or
 $G=Z(G)\mathrm{SL}(2,5)$.
\end{lem}
\begin{proof} Let $S=Sol(G)$. Suppose that $C_G(S)=G$. Thus
$S\leq Z(G)$ and so $S=Z(G)$. Now, consider the central extension
$Z(G)\longrightarrow G\longrightarrow \frac{G}{Z(G)}$. By a
similar argument as in \cite[Lemma 4.2]{AM}, we have that
$K=G'\cap Z(G)$ is of order no more than $2$, $G=G'Z(G)$ and
$\frac{G'}{K}\cong A_5$. Thus \cite[Lemma 4.2]{AM} implies that
there is a subgroup $L$ of $G'$ such that $G'=K\times L$ and
$L\cong A_5$ or $G'\cong \mathrm{SL}(2,5)$. Therefore
$G=G'Z(G)=LKZ(G)=LZ(G)$ and it is clear that $L\cap Z(G)=1$ or
$G=G'Z(G)\cong SL(2,5)Z(G)$. Thus $G\cong A_5\times Z(G)$ or
$G=Z(G)\mathrm{SL}(2,5)$.\\

 Now suppose that $C_G(S)$ is a proper (normal) subgroup of $G$.
 If $C_G(S)$ is  solvable, $C_G(S)\leq S$. Now by \cite[Remark 2.9]{AM}, $\frac{G}{S}=\bigcup^{21}_{i=1} P_i$,
 where $P_1, \dots, P_{21}$ are all  the Sylow subgroups of
 $\frac{G}{S}$. Assume that $P_1, \ldots, P_{10}$
  are Sylow $3$-subgroups, $P_{11}, \ldots, P_{17}$
  are  Sylow $5$-subgroups and $P_{18}, \ldots, P_{21}$
  are Sylow $2$-subgroups of $G$. Now if we choose  any element $a_iS\in P_i\backslash \{1\}$ $(i=1,\ldots, 21)$, then the
  set $\{a_1S, \ldots, a_{21}S\}$ is a maximum clique set for
  $\frac{G}{S}$ and $P_i=C_{\frac{G}{S}}(a_iS)$.\\
For all $i\in\{1, \dots, 10\}$, $|a_iS|=3$ and for $i\in\{11,
\dots, 17\}$, $|a_iS|=5$. Thus
$C_{\frac{G}{S}}(a_iS)=\frac{\langle a_i\rangle S}{S}=\langle
a_iS\rangle$.  Since  $a_i\not\in S$ and for $i\in\{1,\dots,17\}$,
$|a_iS|$ is prime, $a_i\not\in C_G(S)$ for each
$i\in\{1,\dots,17\}$. Thus there exists $s_i\in S$ such that
$a_is_i\neq s_ia_i$ for each $i\in\{1,\dots,17\}$. It is now easy
to see that the set $\{a_i, a_is_i,a_i^2s_i \;|\; i=1, \dots,
10\}\cup \{a_j, a_js_j, a_j^2s_j, a_j^3s_j, a_j^4s_j \;|\; j=11, \dots,
17\}$ is a clique set of $\mathcal{A}_G$. It follows that
$\omega(\mathcal{A}_G)\geq 65$ which is a contradiction.

 Now suppose that $C_G(S)$ is not solvable. Thus $\frac{C_G(S)S}{S}$
is not  solvable and so  $C_G(S)S=G$. Let $N$ be a non-solvable
subgroup of $C_G(S)$ of the least order. It follows that $NS=G$,
 $$\frac{N}{N\cap S} \cong \frac{NS}{S}=\frac{G}{S}\cong A_5,$$ $Sol(N)=N\cap S$ and every
proper subgroup of $N$ is solvable.\\
 If $C_N(Sol(N))=N$, then
$Z(N)=Sol(N)$. By the first part of the proof, $N=Z(N)\times A_5$
or $N=Z(N)\mathrm{SL}(2,5)$ which imply that  $G=SN=S\times A_5$
or $G=S\mathrm{SL}(2,5)$, respectively. If $G=S\times A_5$, then
by \cite[Lemma 2.2]{AM}, $S$ is  abelian. It follows that
$G=SC_G(S)=C_G(S)$, a contradiction, as we are assuming $G\neq
C_G(S)$. Therefore $G=S \mathrm{SL}(2,5)$. Let $\{s_1, s_2, s_3\}$
be a clique of $\mathcal{A}_S$ and $\{b_1Z,b_2Z,\dots,b_{21}Z\}$
be a (maximum) clique set of $\frac{\mathrm{SL}(2,5)}{Z}\cong
A_5$, where $Z=Z(\mathrm{SL}(2,5))$. Then $[b_i, b_j]\not\in Z$,
whenever $i\neq j$ and $i,j\in\{1,2,\ldots,21\}$.  Now
$(b_is_r)(b_js_k)=(b_js_k)(b_is_r)$ if and only if $[b_i,
b_j]=[s_k^{-1}, s_r^{-1}]\in S'\cap SL(2,5)\subseteq Z$, where
$i, j\in\{1,2,\dots,21\}$ and $r, k\in\{1, 2, 3\}$. It follows
that $\{b_1s_i, b_2s_i, \ldots, b_{21}s_i \;|\; i=1, 2, 3\}$ is a
clique set for $\mathcal{A}_{G}$ and so $\omega(\mathcal{A}_G)\geq 63$,
a contradiction . Therefore  $C_N(Sol(N))$ is a proper subgroup
of $N$. Thus $C_N(Sol(N))$ is solvable so that $C_N(Sol(N))\leq
Sol(N)$. Now the same proof as above, gives $\omega(\mathcal{A}_N)\geq
60$, which is a contradiction. This completes the proof.
\end{proof}
\begin{lem}\label{30} Let $G$ be a finite non-solvable group such that
$\omega(\mathcal{A}_G) \leq 57$  and $\frac{G}{Sol(G)}\cong S_5$. Then
$G=G''\langle a\rangle Sol(G)$, where $a^2\in Z(G)$ and $G''\cong
A_5$ or $\mathrm{SL}(2,5)$.
\end{lem}
\begin{proof} Let $S=Sol(G)$. Since $\frac{G}{S}\cong S_5$, it follows that $\frac{G'S}{S}\cong
A_5$ and $|G:G'S|=2$. Note that $S\subseteq Sol(G'S)$ and since
$Sol(G'S)$ is a normal subgroup of $G$, we have $Sol(G'S)=S$. By
the proof of Lemma \ref{15}, $S=Z(G'S)$ and $G'S=S\times A_5$ or
$G'S=S \mathrm{SL}(2,5)$. Thus $G''=A_5$ or $G''=\mathrm{SL}(2,5)$.\\
 Now $G''$ is a non-solvable subgroup of $G'$ and so $\frac{G''S}{S}$ is
a  non-solvable subgroup of $\frac{G'S}{S}\cong A_5$ which
implies that $G''S=G'S$. It follows that  there exists an element
$aS$ of order $2$ in $\frac{G}{S}
 \backslash\frac{G''S}{S}$ and it is easy to see that $G=G''\langle
a\rangle S$ and $a^2\in Z(G)$. This completes the proof.
\end{proof}
\begin{lem}\label{6} Let $G$ be a finite group and
$\omega(\mathcal{A}_G)=\omega(\mathcal{A}_{\frac{G}{Sol(G)}})=57$. Then
$Z(G)=Sol(G)$.
\end{lem}
\begin{proof}By \cite[Lemma 4.3]{AM}, $S=Sol(G)$ is  abelian.
 By Corollary \ref{7},
$\frac{G}{S}\cong \mathrm{PSL}(2,7)$ or $\frac{G}{S}\cong
\mathrm{PGL}(2,7)$. Suppose  for some $x\in G\backslash S$ and
some $a\in S$, that $[a,x]\neq 1$. Then, as
$\omega(\mathcal{A}_G)=57$, by Proposition \ref{50}, there exist
$x_1S, \ldots, x_{56}S \in
 \frac{G}{S}$ such that $T=\{xS, x_1S,\ldots, x_{56}S\}$  is a  clique
  of  $\mathcal{A}_{\frac{G}{S}}$. Now the set $R=\{x, x_1,\ldots, x_{56}\}$ is a maximum clique  of $\mathcal{A}_G$.
  Thus there exists  $x_i\in R$ such that $[x_i,ax]=1$. So $[x_i,x]\in S$, which
  is a contradiction. Therefore $[a,x]=1$ for all $a\in S$  and all
  $x$ in $G\backslash S$. Thus $S\subseteq Z(G)$ and the proof is
  complete.
\end{proof}
\begin{lem}\label{4} Let $G$ be a finite group such that $\omega(\mathcal{A}_G)\leq 57$. If
 there exists a central subgroup $B$ of $G$ of order no
more than $2$ such that $G/B \cong \mathrm{PSL}(2,7)$, then
$G\cong B\times \mathrm{PSL}(2,7)$ or $SL(2,7)$.
\end{lem}
\begin{proof} Since $G/B\cong PSL(2,7)$, it follows that $G=G'B$ and
$G'/(B\cap G')\cong \mathrm{PSL}(2,7)$. Therefore if $G'\cap B=1$,
then $G\cong B\times \mathrm{PSL}(2,7)$. Suppose that $G'\cap B
\neq 1$ so that $|B|=2$. According to the Universal Coefficient
Theorem (see \cite[Theorem 11.4.18 ]{R}) the central extension
 $B \longrightarrow G \longrightarrow G/B $ determines a homomorphism $\delta :
M(G/B)\rightarrow B$ so that $\mathrm{Im}\; \delta=G'\cap B$, where
$M(G/B)$ is the Schur multiplier  of $G/B$ (see \cite[page 354,
Exercise 10]{R}). On the other hand, we know that the Schur multiplier
of $\mathrm{PSL}(2,7)$ is $\mathbb{Z}_2$. Hence $G'\cap B=B$ and
so $B\leq G'$. It follows that $G$ is a perfect group of order
$336$. It is well-known that the only perfect group of order $336$
is $\mathrm{SL}(2,7)$. This completes the proof.
\end{proof}
\begin{lem}\label{8} Let $G$ be a finite non-solvable group such that $\omega(\mathcal{A}_G)\leq 57$
and  $\frac{G}{S}\cong \mathrm{PSL}(2,7)$. Then $G\cong Z(G)\times
\mathrm{PSL}(2,7)$ or $G\cong Z(G)\mathrm{SL}(2,7)$.
\end{lem}
\begin{proof} Since $57=\omega(\mathcal{A}_{\frac{G}{S}})\leq \omega(\mathcal{A}_G)\leq 57$, we have  $\omega(\mathcal{A}_{\frac{G}{S}})=\omega(\mathcal{A}_G)=57$.
By Lemma \ref{6},  $S=Z(G)$. Now by the same argument as in Lemma
\ref{4}, considering the central extension $Z(G)\rightarrow
G\rightarrow\frac{G}{Z(G)}$ we have that $K=G'\bigcap Z(G)$ is of
order no more than $2$, $G=G'Z(G)$ and $\frac{G'}{K}\cong
\mathrm{PSL}(2,7)$. Thus Lemma \ref{4} implies that there is a
subgroup $L$ of $G'$ such that $G'=K\times L$ or $G'\cong
\mathrm{SL}(2,7)$ and $L\cong \mathrm{PSL}(2,7)$. Now if
$G'=K\times L$, then $G=G'Z(G)=KLZ(G)=LZ(G)$ and it is clear that
$L\cap Z(G)=1$. So $G=L\times Z(G)\cong \mathrm{PSL}(2,7)\times
Z(G)$. Otherwise $G'\cong \mathrm{SL}(2,7)$, and so $G\cong
Z(G)\mathrm{SL}(2,7)$.
\end{proof}
\begin{lem}\label{35} Let $G$ be a  finite non-solvable group such that $\omega(\mathcal{A}_G)\leq 57$
and  $\frac{G}{S}\cong \mathrm{PGL}(2,7)$. Then $G=G''\langle
a\rangle Z(G)$, where $a^2\in Z(G)$ and $G''\cong
\mathrm{PSL}(2,7)$ or $G''\cong \mathrm{SL}(2,7)$.
\end{lem}
\begin{proof} Since $57=\omega(\mathcal{A}_{\frac{G}{S}})\leq \omega(\mathcal{A}_G)\leq 57$,  we have $\omega(\mathcal{A}_{\frac{G}{S}})=\omega(\mathcal{A}_G)=57$.
By Lemma \ref{6}, $S=Z(G)$ and so $\frac{G}{Z(G)}\cong
\mathrm{PGL}(2,7)$. Thus  $\frac{G'Z(G)}{Z(G)}\cong
\mathrm{PSL}(2,7)$ and it follows from Lemma \ref{8} that
$G'Z(G)=Z(G)\times \mathrm{PSL}(2,7)$ or
$G'Z(G)=Z(G)\mathrm{SL}(2,7)$ and $|\frac{G}{Z(G)}
:\frac{G'Z(G)}{Z(G)}|=2$. Thus $G''\cong \mathrm{PSL}(2,7)$ or
$G''\cong \mathrm{SL}(2,7)$. Suppose that $aZ(G)$ is an element of
$\frac{G}{Z(G)}\setminus \frac{G'Z(G)}{Z(G)}$ of order $2$. Then
$G=G''\langle a\rangle Z(G)$, where $a^2\in Z(G)$ and $G''\cong
\mathrm{PSL}(2,7)$ or $G''\cong \mathrm{SL}(2,7)$.
\end{proof}
\noindent{\bf Proof of Theorem \ref{1.3}.} This follows from Lemmas
\ref{15}, \ref{30}, \ref{8}, and \ref{35}.
\section{clique numbers of the non-commuting graphs of the minimal simple groups}
For a non-trivial abelian group $A$, we define
$\omega(\mathcal{A}_A)=1$.
\begin{lem}\label{20} Let $G$ be a group such that there exist non-trivial subgroups $A_1,\dots,A_n$ of $G$ with $G=\bigcup_{i=1}^nA_i$ and
$A_i \cap A_j=Z(G)$ for $i\neq j$.
\begin{enumerate}
\item If $C_G(g)\leq A_i$ for all
$g\in A_i\backslash Z(G)$, then
 $\omega(\mathcal{A}_G)=\sum_{i=1}^n \omega(\mathcal{A}_{A_i})$.
\item If every clique of $\mathcal{A}_{A_i}$ can be extended  to a maximum
clique  of $\mathcal{A}_{A_i}$ for each $i\in\{1,\dots,n\}$, then
the same property for $\mathcal{A}_G$ is true. In particular, if
all $A_i$'s are either abelian  or  $AC$-groups, the mentioned
property holds for $\mathcal{A}_G$.
\end{enumerate}
\end{lem}
\begin{proof}
(1) \; If $X$ is any clique of $\mathcal{A}_G$, then
$X=\bigcup_{i=1}^nX_i$, where $X_i\subset A_i\backslash Z(G)$ for
each $i\in\{1,\dots,n\}$. By hypothesis, $|X|=\sum_{i=1}^n |X_i|$
and since $|X_i|\leq \omega(\mathcal{A}_{A_i})$, it follows that
$|X|\leq \sum_{i=1}^n \omega(\mathcal{A}_{A_i})$. Now let $W_i$ be a
maximum clique
 of $\mathcal{A}_{A_i}$ for each $i\in\{1,\dots,n\}$. We claim that $W=\displaystyle\bigcup_{i=1}^n W_i$ is a
maximum clique  for $\mathcal{A}_G$. Suppose, for a
contradiction,  that there exist two distinct commuting elements
$a$ and $b$ in $\displaystyle\bigcup_{i=1}^n W_i$. Thus there
exist $i\neq j$ such that $a\in A_i$ and $b\in A_j$. Therefore
$a\in C_G(b)\leq A_j$ and so $a\in A_i\cap A_j=Z(G)$, which is
impossible. Since $|W|=\sum_{i=1}^n \omega(\mathcal{A}_{A_i})$, the
proof of (1) is complete.\\
(2) \; It is straightforward.\\
\end{proof}

{\noindent \bf Proof of Theorem \ref{1.4}.} (i) The Suzuki group
$G$ contains subgroups $F$, $A$, $B$ and $C$ such that $|F|=q^2$,
$|A|=q-1$, $|B|=q-2r+1$ and $|C|=q+2r+1$ (see \cite[Chapter XI,
Theorems 3.10 and 3.11]{HB}). Also by \cite[pp. 192-193, Theorems
3.10 and 3.11]{HB}, the conjugates of $A$, $B$, $C$ and $F$  in
$G$ form a partition for $G$, and $A$, $B$, $C$ are cyclic. These
subgroups are all centralizers of some elements in $G$ and $F$ is
a Sylow 2-subgroup of $G$.\\ Now \cite[Chapter XI, Theorems 3.10
and 3.11]{HB} implies that the number of conjugates of $C$, $B$,
$A$ and $F$ in $G$ are respectively,
$\alpha=\frac{q^2(q-1)(q^2+1)}{4(q+2r+1)}$,
$\beta=\frac{q^2(q-1)(q^2+1)}{4(q-2r+1)}$,
$\gamma=\frac{q^2(q^2+1)}{2}$  and  $\delta=q^2+1$
 and also
$$G=\cup_{i=1}^{\delta} C_G(f_i) \bigcup  \cup_{i=1}^{\gamma} C_G(a_i)
\bigcup  \cup_{i=1}^{\beta} C_G(b_i)  \bigcup  \cup_{i=1}^{\alpha}
C_G(c_i),$$
 Now by \cite[Chapter XI, proof of Lemma 5.9]{HB}, $|C_F(g):Z(F)|=2$, for all $g\in F\backslash Z(F)$. If $C_F(g)=H$, then
 $|\frac{H}{Z(H)}|=2$ which implies that  $H$ is abelian. It follows that $F$ is an AC-group.
 Let $\{a_1,a_2,\dots,a_n\}$ be a clique  of $\mathcal{A}_F$. Then
 $F=C_F(a_1)\cup \cdots \cup C_F(a_n)$
and the set $\{ \frac{C_F(a_i)}{Z(F)} \;|\; i=1,2,\ldots ,n\}$
forms a partition for $\frac{F}{Z(F)}$. Thus
$\omega(\mathcal{A}_F)=q-1$.
 Now it follows from  Lemma \ref{20}, that
  $$\omega(\mathcal{A}_{\mathrm{Sz}(q)})=(q^2+1)(q-1)+\frac{q^2(q^2+1)}{2}+\frac{q^2(q^2+1)(q-1)}{4(q+2r+1)}
+\frac{q^2(q^2+1)(q-1)}{4(q-2r+1)}.$$
(ii) \; It follows from Lemma \ref{20} and the proof of part (i). $\hfill \Box$\\

{\noindent\bf Proof of Theorem \ref{26}.}
  Let $G=\mathrm{PSL}(3,3)$. It is easy to see (e.g., by {\sf
GAP} \cite{GAP}) that the set of order elements of $G$ is
$\{1,2,3,4,6,8,13\}$ and if $A=\{C_G(g) \;|\; g\in G, \;
|C_G(g)|=6 \}$, $B=\{ C_G(g) \;|\; g\in G, \; |C_G(g)|=8 \}$,
$C=\{C_G(g) \;|\; g\in G, \; |C_G(g)|=9\}$ and $D=\{C_G(g) \;|\;
g\in G, \; |C_G(g)|=13\}$, then $|A|=468$, $|B|=351$, $|C|=104$
and $|D|=144$. Also we know that if $|C_G(g)|\in \{6,8\}$, then
$C_G(g)$ is a cyclic subgroup of $G$ and so there exists $a\in G$
such that $C_G(g)=\langle a \rangle$. It follows that   $\langle
a \rangle=C_G(a)=C_G(g)$. Thus
 there exist elements $a_i,b_j,c_k,d_\ell\in G$ such that
$|C_G(a_i)|=6$ for $1\leq i\leq 468$, $|C_G(b_i)|=8$ for $1\leq
i\leq 351$, $|C_G(c_i)|=9$ for $1\leq i\leq 104$ and
$|C_G(d_i)|=13$ for $1\leq i\leq 144$. Now it is easy to see
(e.g., by GAP \cite{GAP}) that $$G=\bigcup_{x\in X} C_G(x),$$
where $X=\{a_1, \ldots, a_{468}, b_1, \dots, b_{351}, c_1, \dots,
c_{104}, d_1, \ldots, d_{144}\}$. Since the set of order elements
of $G$ is $\{1,2,3,4,6,8,13\}$, it follows  that $X$ is a clique
for $\mathcal{A}_G$. Also since for all $x\in X$, $C_G(x)$ is
abelian, we have $\omega(\mathcal{A}_G)=|X|=468+351+104+144=1067$.
This completes the proof. $\hfill \Box$\\

\begin{rem} It is not true that every clique of
the non-commuting graph of $G=\mathrm{PSL}(3,3)$ can be extended
to a maximum clique. It can be seen that there are two distinct
elements $x_1,x_2\in X$ such that $C_G(x_1)\cap C_G(x_2)$
contains a non-trivial element $a$. Now $\{a\}$ cannot be
extended to a maximum clique. On the other hand it is easy to see that every clique containing only elements of orders in $\{6,8,13\}$
 can be extended to a maximum clique. We leave the easy
proof to the reader.
\end{rem}
{\noindent\bf Acknowledgements.} The authors are grateful to the referees for their very helpful comments.
The research of the first and third authors were supported by the Center of Excellence for Mathematics, University of Isfahan. The first author's research was in part supported by a grant from IPM (No. 87200118).

\end{document}